\documentclass[11pt]{article} 
\usepackage{amsfonts,amsmath,latexsym,amssymb,mathrsfs,amsthm,comment}
\usepackage[colorlinks=true, linkcolor=black, citecolor=black]{hyperref}
\usepackage{graphicx}
\usepackage{xcolor}
\usepackage{booktabs}

\let\OLDthebibliography\thebibliography
\renewcommand\thebibliography[1]{
  \OLDthebibliography{#1}
  \setlength{\parskip}{1pt}
  \setlength{\itemsep}{0pt plus 0.0ex}
}

\def\numberlikeadb{\global\def\theequation{\thesection.\arabic{equation}}}
\numberlikeadb
\newtheorem{theorem}{Theorem}[section]
\newtheorem{lemma}[theorem]{Lemma}
\newtheorem{corollary}[theorem]{Corollary}

\newtheorem{proposition}[theorem]{Proposition}
\newtheorem{remark}[theorem]{Remark}

\usepackage{color}

\usepackage{lscape}
\usepackage{caption}
\usepackage{multirow}
\allowdisplaybreaks
\begin{document}

\title{The distribution of the ratio of products of independent zero mean normal random variables}
	\author{Robert E. Gaunt\footnote{Department of Mathematics, The University of Manchester, Oxford Road, Manchester M13 9PL, UK, robert.gaunt@manchester.ac.uk; heather.sutcliffe@manchester.ac.uk}\:\, and Heather L. Sutcliff$\mathrm{e}^{*}$}

	\date{} 
	\maketitle
	
	\vspace{-7mm}

\begin{abstract}    
	Let $X_1,\ldots,X_M$ and  $Y_1,\ldots,Y_N$ be independent zero mean normal random variables with variances $\sigma_{X_i}^2$, $i=1,\ldots,M$, and $\sigma_{Y_j}^2$, $j=1,\ldots,N$, respectively, and let $X=X_1\cdots X_M$ and $Y=Y_1\cdots Y_N$. In this paper, we derive the exact probability density function of the ratio $X/Y$. We  apply this formula to derive exact formulas for the cumulative distribution function and the characteristic function. We also obtain further distributional properties, including asymptotic approximations for the probability density function, tail probabilities and the quantile function. 
    %In the final section, we list some of the PDFs which arise as special cases of our distribution.
	%Some basic distributional properties are also derived, including identification of parameter regimes under which the density is bounded, asymptotic approximations of tail probabilities, and fractional moments; in particular, we see that the mean is undefined. In the case that X and Y are independent symmetric variance-gamma random variables, an exact formula is also given for the cumulative distribution function of the ratio X /Y .
\end{abstract} 



\section{Introduction}

Ratios and products of random variables arise in applications throughout probability and statistics, as well as the biological and physical sciences; for some examples of application areas see \cite{nd06}. As such, there is now an extensive literature concerning the distributional theory of ratios and products of random variables, with exact formulas derived for many classical probability distributions; see \cite{nad07} for an overview of some of this literature. 

In this paper, we study the exact distribution of the ratio of products of independent zero mean normal random variables, and in doing so fill a clear gap in the literature. More specifically, let $X_1,\ldots,X_M$ and  $Y_1,\ldots,Y_N$ be independent normal random variables with zero mean and variances $\sigma_{X_i}^2$, $i=1,\ldots,M$, and $\sigma_{Y_j}^2$, $j=1,\ldots,N$, respectively. Also, let $\sigma_{M}=\sigma_{X_1}\cdots\sigma_{X_M}$ and $s_N=\sigma_{Y_1}\cdots\sigma_{Y_N}$. We use the convention that the empty product is set to one. We thus set $\sigma_0=1$ and $s_0=1$, and set $X=1$ when $M=0$, and $Y=1$ when $N=0$.  In this paper, we study the distribution of the ratio $Z=X/Y$, deriving a number of its key distributional properties, including the probability density function (PDF), cumulative distribution function (CDF), characteristic function, fractional moments, and also derive asymptotic approximations for the PDF, tail probabilities and the quantile function. These results are stated and proved in Section \ref{sec2}.

%Given the importance of the normal distribution, it is not surprising that special cases 

%Before deriving our results, we close this introduction by reviewing 

The general formulas for the PDF, CDF and characteristic function of Section \ref{sec2} are stated in terms of the Meijer $G$-function (defined in Appendix \ref{appa}). In Section \ref{sec3}, we present some special cases in which these formulas take a simpler form.
%; as a by-product of this analysis we deduce a new reduction formula for the Meijer $G$-function, which we state in Section \ref{sec4}. 
We now briefly review some special cases that have already been treated in the literature. A closed-form formula for the PDF of the product $X=X_1\cdots X_M$, expressed in terms of the Meijer $G$-function, was obtained by \cite{product normal}: 
%for $x\in\mathbb{R}$,
\begin{align}
	f_X(x)=\frac{1}{(2\pi)^{M/2}\sigma_M}G^{M,0}_{0,M}\bigg(\frac{x^2}{2^M\sigma^2_M} \; \bigg| \;{- \atop 0,\ldots,0}\bigg),
    \quad x\in\mathbb{R}.
    \label{1.1}
\end{align}
In the case $M=2$, the PDF (\ref{1.1}) for the product $X=X_1X_2$ takes the simpler form $f_X(x)=(\pi\sigma_2)^{-1}K_0(|x|/\sigma_2)$, $x\in\mathbb{R}$, where $K_0(x)$ is a modified Bessel function of the second kind (see \cite{wb32}). However, further simplifications are unavailable for $M\geq3$. It is well-known that the ratio of independent zero-mean normal random variables follows the Cauchy distribution, and so in the case $M=N$, the ratio $Z=X/Y$ is distributed as the product of $M$ independent Cauchy random variables. The PDF of the product of two independent Cauchy random variables was derived by \cite{r65}, and more generally an expression for the PDF of the product of independent Cauchy random variables was obtained by \cite{st66}. Exact formulas for the PDF and CDF of $M$ independent Cauchy random variables also follow as special cases of the formulas of \cite{nad} for the product of $M$ independent Student's $t$ distributed random variables, which are expressed in terms of the Meijer $G$-function. 

%Other distributional properties of these special case distributions are scattered across the literature (we will point to references in Section \ref{sec2}) which can be difficult to track down, and there are also obvious omissions from the literature such as asymptotic approximations for the quantile functions. In addition to the new results we obtain for the ratio $Z=X/Y$ for general $M,N$, a helpful contribution of this paper will be to collect in a single reference a number of the most important distributional properties of product of $M$ independent zero mean normal random variables and the product of $M$ independent Cauchy random variables, as such results simply follow on setting $N=0$ and $N=M$ in our general results.

%The results of this work will put 

%where $\sigma_{M}=\sigma_{X_1}\cdots\sigma_{X_M}$ and $\sigma_{0}$ is set to 1. Further, let $Y_1,\ldots,Y_N$ be independent random variables with $Y\sim N(0,\sigma_{Y_j}^2)$, $1\leq j\leq N$ and define $s_N=\sigma_{Y_1}\cdots\sigma_{Y_N}$, where $s_0$ is set to 1. In this paper we study the distribution of the ratio $X/Y$, which we call the normal product ratio distribution. 
%In the case $M=N=1,$ $Z$ has the Cauchy distribution
\section{The normal product ratio distribution}\label{sec2}

\subsection{Probability density function}
In the following theorem, we provide a closed-form formula for the PDF $f_Z(z)$ of the ratio $Z=X/Y$ which is expressed in terms of the Meijer $G$-function.
\begin{theorem}Suppose $M,N\geq0$ and $M+N\geq1$. Let $X_1,\ldots,X_M$ and $Y_1,\ldots,Y_N$ be independent random variables where $X_i\sim N(0,\sigma^2_{X_i})$, $i=1,\ldots,M$, and $Y_j\sim N(0,\sigma^2_{Y_j})$, $j=1,\ldots,N$. Let $X=X_1\cdots X_M$ and $Y=Y_1\cdots Y_N$, and denote $Z=X/Y$. Then, for $z\in\mathbb{R}$,
	\begin{align}
		f_Z(z)=\frac{2^{(N-M)/2}s_N}{\pi^{(M+N)/2}\sigma_M}G^{M,N}_{N,M}\bigg(\frac{2^Ns_N^2}{2^M\sigma_M^2}z^2 \; \bigg| \;{0,\ldots,0 \atop 0,\ldots,0}\bigg).\label{2.1}
	\end{align}
\end{theorem}
\begin{comment}
\begin{remark}
	When $M=0$, we have $Z=\frac{1}{Y_1\cdots Y_N}$ and the PDF of $Z$ is then 
	\begin{align}
		f_Z(z)=\frac{2^{N/2}\sigma_Y}{\pi^{N/2}}G^{0,N}_{N,0}\bigg(2^N\sigma_Y^2z^2 \; \bigg| \;{0,\ldots,0 \atop -}\bigg),\quad z\in\mathbb{R}.\label{2.3}
		\end{align}
\end{remark}
\end{comment}
\begin{proof}
When $M\geq1$ and $N\geq1$ we have, for $z\in\mathbb{R}$,
	\begin{align}
		f_Z(z)&=\int_{-\infty}^{\infty}|x|f_X(zx)f_Y(x)\,\mathrm{d}x=2\int_{0}^{\infty}xf_X(zx)f_Y(x)\,\mathrm{d}x\nonumber\\
		&=\frac{2}{(2\pi)^{(M+N)/2}\sigma_Ms_N}\int^{\infty}_{0}x\,G^{M,0}_{0,M}\bigg(\frac{z^2x^2}{2^M\sigma_M^2} \; \bigg| \;{- \atop 0,\ldots,0}\bigg)G^{N,0}_{0,N}\bigg(\frac{x^2}{2^Ns_N^2} \; \bigg| \;{-\atop 0,\ldots,0}\bigg)\,\mathrm{d}x\nonumber\\
		&=\frac{1}{(2\pi)^{(M+N)/2}\sigma_Ms_N}\int^{\infty}_{0}G^{M,0}_{0,M}\bigg(\frac{z^2u}{2^M\sigma_M^2} \; \bigg| \;{-\atop 0,\ldots,0}\bigg)G^{N,0}_{0,N}\bigg(\frac{u}{2^Ns_N^2} \; \bigg| \;{- \atop 0,\ldots,0}\bigg)\,\mathrm{d}u,\nonumber
	\end{align}
	where in the third step we used the PDF formula (\ref{1.1}).
    %and in the fourth step we used a change of variables. 
    On using formula 5.6.2(1) of \cite{luke} to evaluate the integral we obtain (\ref{2.1}).
	When $N=0$ our PDF formula (\ref{2.1}) reduces to the PDF (\ref{1.1}) for $X=X_1\cdots X_M$, as given in \cite{product normal}, and so the formula is also valid for $M\geq1$ and $N=0$. Finally, we verify that formula (\ref{2.1}) is also valid for $M=0$ and $N\geq1$. In this case, by the standard formula for the PDF of a one-to-one transformation of a random variable, $f_Z(z)=f_{1/Y}(z)=z^{-2}f_Y(1/z)$, $z\in\mathbb{R}$, so that, by formula (\ref{1.1}),
    \begin{align*}
     f_Z(z)&=\frac{1}{(2\pi)^{N/2}s_N}\frac{1}{z^2}G^{N,0}_{0,N}\bigg(\frac{1}{2^Ns^2_Nz^2} \; \bigg| \;{- \atop 0,\ldots,0}\bigg) =\frac{2^{N/2}s_N}{\pi^{N/2}}G^{0,N}_{N,0}\bigg(2^Ns_N^2z^2 \; \bigg| \;{0,\ldots,0 \atop -}\bigg),
    \end{align*}
  where the second equality follows from applying the identities (\ref{meijergidentity}) and (\ref{mgiv}). This confirms that formula (\ref{2.1}) holds for $M=0$ and $N\geq1$.
\end{proof}

\subsection{Cumulative distribution function}
Let $F_Z(z)=\mathbb{P}(Z\leq z)$ denote the CDF of $Z=X/Y$. In the following theorem, we provide closed-form formulas for the CDF of $Z$.
\begin{theorem}\label{thmcdf}Let $M,N\geq0$ and $M+N\geq1$. Then, for $z\in\mathbb{R}$,
	\begin{align}F_Z(z)&=\frac{1}{2}+\frac{2^{(N-M-2)/2}s_N}{\pi^{(M+N)/2}\sigma_M}z\,G^{M,N+1}_{N+1,M+1}\bigg(\frac{2^Ns_N^2}{2^M\sigma_M^2}z^2 \; \bigg| \;{\frac{1}{2},0,\ldots,0 \atop 0,\ldots,0,-\frac{1}{2}}\bigg)\label{2.2a} \\
		&=\frac{1}{2}+\frac{\mathrm{sgn}(z)}{2\pi^{(N+M)/2}}G^{M,N+1}_{N+1,M+1}\bigg(\frac{2^Ns_N^2}{2^M\sigma_M^2}z^2 \; \bigg| \;{1,\frac{1}{2},\ldots,\frac{1}{2}\atop \frac12,\ldots,\frac12,0}\bigg)\label{2.2b},
	\end{align}
    where $\mathrm{sgn}(z)$ is the sign function, $\mathrm{sgn}(z)=-1$ for $z<0$, $\mathrm{sgn}(0)=0$, $\mathrm{sgn}(z)=1$ for $z>0$.
\end{theorem}
\begin{proof}As the PDF (\ref{2.1}) is symmetric about the origin, we have that, for $z>0$,
	\begin{align}
		F_Z(z)&=\frac{1}{2}+\frac{2^{(N-M)/2}s_N}{\pi^{(M+N)/2}\sigma_M}\int_{0}^{z}G^{M,N}_{N,M}\bigg(\frac{2^Ns_N^2}{2^M\sigma_M^2}x^2 \; \bigg| \;{0,\ldots,0 \atop 0,\ldots,0}\bigg)\,\mathrm{d}x\nonumber\\
		&=\frac{1}{2}+\frac{1}{2\pi^{(M+N)/2}}\int_{0}^{\frac{2^Ns_N^2}{2^M\sigma_M^2}z^2}u^{-1/2}G^{M,N}_{N,M}\bigg(u \; \bigg| \;{0,\ldots,0 \atop 0,\ldots,0}\bigg)\,\mathrm{d}u\nonumber\\
		&=\frac{1}{2}+\frac{1}{2\pi^{(M+N)/2}}\bigg[u^{1/2}G^{M,N+1}_{N+1,M+1}\bigg(u \; \bigg| \;{\frac12,0,\ldots,0 \atop 0,\ldots,0,-\frac12}\bigg)\bigg]^{\frac{2^Ns_N^2}{2^M\sigma_M^2}z^2}_{0}%\label{2.21}
    \nonumber    \\
		&=\frac{1}{2}+\frac{2^{(N-M-2)/2}s_N}{\pi^{(M+N)/2}\sigma_M}\,z\,G^{M,N+1}_{N+1,M+1}\bigg( \frac{2^Ns_N^2}{2^M\sigma_M^2}z^2\; \bigg| \;{\frac12,0,\ldots,0 \atop 0,\ldots,0,-\frac12}\bigg)\nonumber,
	\end{align}
	where the integral was evaluated using (\ref{mint}). We also used that the $G$-function in the third equality evaluated at $u=0$ is equal to zero, which can be seen from the contour integral representation (\ref{mdef}) of the $G$-function. This confirms that formula (\ref{2.2a}) holds for $z>0$; the verification for $z<0$ is similar and is hence omitted. Lastly, formula (\ref{2.2b}) follows from (\ref{2.2a}) by an application of the identity (\ref{meijergidentity}), in which the cases $z>0$ and $z<0$ are treated separately.
\end{proof}

\begin{remark} When specialised to the case $N=0$, the formulas (\ref{2.2a}) and (\ref{2.2b}) represent new formulas for the CDF of the product $X=X_1\cdots X_M$. An alternative formula for the CDF of $X$, also expressed in terms of the Meijer $G$-function, is given in \cite[Lemma 2.1]{s17}.
%Formulas, expressed in terms of the Meijer $G$-function, for the CDF of $X=X_1\cdots X_M$ are given in \cite[Lemma 2.1]{s17}
%; using basic identities for the Meijer $G$-function their formulas can be seen to be equivalent to those of Theorem \ref{thmcdf} in the special case $N=0$. 
\end{remark}

\subsection{Characteristic function}
In the following theorem, we provide a closed-form formula for the characteristic function of $Z$, which we denote by $\varphi_Z(t)=\mathbb{E}[\mathrm{e}^{\mathrm{i}tZ}]$. 
\begin{theorem}
(i) Suppose $M\geq1$ and $N\geq0$. Then, for $t\in\mathbb{R}$, 
	\begin{align}
		\varphi_Z(t)
        %&=\frac{2^{(N-M+2)/2}s_N}{\pi^{(N+M-1)/2}\sigma_M|t|}G^{M-1,N+1}_{N+1,M-1}\bigg(\frac{2^{N+2}s_N^2}{2^M\sigma_M^2}\frac{1}{t^2} \; \bigg| \;{\frac{1}{2},0,\ldots,0\atop 0,\ldots,0}\bigg)\label{2.3a}\\
		=\frac{1}{\pi^{(N+M-1)/2}}G^{N+1,M-1}_{M-1,N+1}\bigg(\frac{2^M\sigma_M^2t^2}{2^{N+2}s_N^2} \; \bigg| \;{\frac{1}{2},\ldots,\frac{1}{2} \atop 0,\frac{1}{2},\ldots,\frac{1}{2}}\bigg).\label{2.3b}
	\end{align}
(ii) Suppose $M=0$ and $N\geq1$. Then, for $t\in\mathbb{R}$, 
\begin{align}\label{ratioo}
\varphi_Z(t)
        %&=\frac{2^{(N-M+2)/2}s_N}{\pi^{(N+M-1)/2}\sigma_M|t|}G^{M-1,N+1}_{N+1,M-1}\bigg(\frac{2^{N+2}s_N^2}{2^M\sigma_M^2}\frac{1}{t^2} \; \bigg| \;{\frac{1}{2},0,\ldots,0\atop 0,\ldots,0}\bigg)\label{2.3a}\\
		=\frac{1}{\pi^{(N-1)/2}}G^{N+1,0}_{0,N+2}\bigg(\frac{t^2}{2^{N+2}s_N^2} \; \bigg| \;{- \atop 0,\frac{1}{2},\ldots,\frac{1}{2}}\bigg).   
\end{align}
\end{theorem}
\begin{proof}
(i) Suppose $M\geq1$ and $N\geq0$.	Since the PDF (\ref{2.1}) is symmetric about the origin we have that, for $t\in\mathbb{R}$,
	\begin{align}
		\varphi_Z(t)=\mathbb{E}[\cos(tZ)]&=
		%2\int_{0}^{\infty}\frac{2^{(N-M)/2}\sigma_Y}{\pi^{(N+M)/2\sigma_X}}\cos(tz)G^{M,N}_{N,M}\bigg(\frac{2^N\sigma_Y^2}{2^M\sigma_X^2}z^2 \; \bigg| \;{0,\ldots,0 \atop 0,\ldots,0}\bigg)\,\mathrm{d}z\nonumber\\
		\frac{2^{(N-M+2)/2}s_N}{\pi^{(N+M)/2}\sigma_M}\int_{0}^{\infty}\cos(tz)G^{M,N}_{N,M}\bigg(\frac{2^Ns_N^2}{2^M\sigma_M^2}z^2 \; \bigg| \;{0,\ldots,0 \atop 0,\ldots,0}\bigg)\,\mathrm{d}z
\nonumber\\
&=\frac{2^{(N-M+2)/2}s_N}{\pi^{(N+M-1)/2}\sigma_M|t|}G^{M,N+1}_{N+2,M}\bigg(\frac{2^{N+2}s_N^2}{2^M\sigma_M^2}\frac{1}{t^2} \; \bigg| \;{\frac12,0,\ldots,0 \atop 0,\ldots,0}\bigg)\nonumber\\
&=\frac{2^{(N-M+2)/2}s_N}{\pi^{(N+M-1)/2}\sigma_M|t|}G^{M-1,N+1}_{N+1,M-1}\bigg(\frac{2^{N+2}s_N^2}{2^M\sigma_M^2}\frac{1}{t^2} \; \bigg| \;{\frac{1}{2},0,\ldots,0\atop 0,\ldots,0}\bigg)\nonumber\\
&=\frac{1}{\pi^{(N+M-1)/2}}G^{M-1,N+1}_{N+1,M-1}\bigg(\frac{2^{N+2}s_N^2}{2^M\sigma_M^2}\frac{1}{t^2} \; \bigg| \;{1,\frac{1}{2},\ldots,\frac{1}{2} \atop \frac{1}{2},\ldots,\frac{1}{2}}\bigg)\label{2.3a},
\end{align}
where the integral was evaluated using formula 5.6.3(18) of \cite{luke}, the penultimate equality was obtained by using (\ref{lukeformula}) and the final equality was obtained using (\ref{meijergidentity}). Finally, we obtain (\ref{2.3b}) from (\ref{2.3a}) by using (\ref{mgiv}).

\vspace{2mm}

\noindent (ii) The proof is similar to part (i), with the only difference being that we do not use identity (\ref{lukeformula}) to reduce the $G$-function to one of lower order (we cannot do this here as $M=0$).
\end{proof}

\begin{remark}
As can been seen from the asymptotic behaviour of the PDF of $Z$ as $|z|\rightarrow\infty$ (as described in Theorem \ref{thmtail} below), the moment generating function $M_Z(t)=\mathbb{E}[\mathrm{e}^{tZ}]$ is not defined for $N\geq1$. The moment generating function is only defined in the cases $M=1$, $N=0$ (for all $t\in\mathbb{R}$) and $M=2$, $N=0$ (for $|t|<1/\sigma_2$); this can be seen from the limiting form (\ref{2.52}).    
\end{remark}    

\subsection{Asymptotic behaviour of the distribution}
%The Meijer $G$-function in the PDF formula (\ref{2.1}) makes it difficult to parse at first inspection and it is therefore useful to study the asymptotic behaviour of the distribution. In this section we derive asymptotic approximations for the PDF, tail probabilities and the quantile function.

In the following theorems, we present asymptotic approximations for the PDF, CDF and quantile function of the ratio $Z=X/Y$. We let $\overline{F}_Z(z)=\mathbb{P}(Z> z)$ denote the complementary CDF of $Z$. Since the distribution of $Z$ is symmetric about the origin, $F_Z(-z)=\overline{F}_Z(z)$, so we only provide an asymptotic approximation for $\overline{F}_Z(z)$ as $z\rightarrow\infty$. Also, for $0<p<1$, we denote the quantile function of $Z$ by $Q(p)=F_Z^{-1}(p)$. Again, due to the symmetry of the distribution, $Q(1-p)=-Q(p)$, and so we only provide asymptotic approximations for $Q(p)$ as $p\rightarrow0$. The proofs of the theorems are deferred until the end of this subsection.

\begin{theorem}\label{thm2.4}
(i) Suppose $M\geq1$, $N\geq0$. Then, as $z\rightarrow0$,
\begin{equation}
	f_Z(z)\sim\frac{2^{(M+N-2)/2}s_N}{\pi^{(M+N)/2}\sigma_M}\frac{(-\ln|z|)^{M-1}}{(M-1)!}.\label{plim1}
\end{equation}
%In this case, the distribution of $Z$ is unimodal with mode 0.
%\vspace{2mm}
%\noindent 
(ii) Suppose $M=0$, $N\geq1$. Then, as $z\rightarrow0$,
\begin{equation}
	f_Z(z)\sim \frac{2^{(N-2)/2}}{\sqrt{N\pi}s_N^{1/N}}|z|^{-(N+1)/N}\exp\bigg(-\frac{N}{2(s_N|z|)^{2/N}}
    %N\bigg(\frac{1}{2^{N/2}s_N|z|}\bigg)^{2/N}
    \bigg).\label{2.42}
\end{equation}
%In this case, the distribution of $Z$ is bimodal with modes $\pm$.
\end{theorem}

\begin{theorem}\label{thmtail}
 Suppose $M\geq0$, $N\geq1$. Then, as $|z|\rightarrow\infty$,
	\begin{align}
		f_Z(z)\sim\frac{2^{(M+N-2)/2}\sigma_M}{\pi^{(N+M)/2}(N-1)!s_N}\frac{(\ln|z|)^{N-1}}{z^2}.\label{pdflim}
	\end{align}
 Also, as $z\rightarrow\infty$,
 \begin{align}\label{CDFf1}
		\overline{F}_Z(z)\sim\frac{2^{(M+N-2)/2}\sigma_M}{\pi^{(M+N)/2}(N-1)!s_N}\frac{(\ln(z))^{N-1}}{z}.
	\end{align}
\end{theorem}

\begin{theorem}\label{qprop}
	(i) Suppose $M\geq0$ and $N\geq1$. Then, as $p\rightarrow0$,
	\begin{align}
		%&Q(p)\sim\frac{2^{(M+N-2)/2}\sigma_M}{\pi^{(M+N)/2}s_N(N-1)!}\frac{\big(\ln(1/(1-p))\big)^{N-1}}{1-p}\label{limf1}, \quad p\rightarrow1\\
		Q(p)\sim -\frac{2^{(M+N-2)/2}\sigma_M}{\pi^{(M+N)/2}s_N(N-1)!}\frac{\big(\ln(1/p)\big)^{N-1}}{p}.\label{limf2}
	\end{align}
	(ii) Suppose $M\geq1$ and $N=0$. Then, as $p\rightarrow0$, 
	\begin{align}
		%Q(p)&\sim \bigg(\frac{2\sigma_M^{M/2}}{M}\ln\bigg(\frac{1}{1-p}\bigg)\bigg)^{M/2},\quad p\rightarrow1,\label{limf3}\\
		Q(p)\sim-\bigg(\frac{2\sigma_M^{M/2}}{M}\ln\big(1/p\big)\bigg)^{M/2}.\label{limf4}
	\end{align}
\end{theorem}

\begin{remark} 
	When $M\geq1$, $N=0$, we have $Z=X=X_1\cdots X_M$. In this case, the following limiting forms for the PDF and tail probabilities were obtained by \cite{gaunt bern} and \cite{Leipus}, respectively:
	\begin{align}
		f_Z(z)&\sim \frac{2^{M/2-1}}{\sqrt{M\pi}\sigma_M^{1/M}}|z|^{1/M-1}\exp\bigg(-\frac{M}{2}\bigg(\frac{|z|}{\sigma_M}\bigg)^{2/M}\bigg), \quad |z|\rightarrow\infty,\label{2.52} \\
      \overline{F}_Z(z)&\sim\frac{2^{M/2-1}\sigma_M^{1/M}}{\sqrt{M\pi}}z^{-1/M}\exp\bigg(-\frac{M}{2}\bigg(\frac{z}{\sigma_M}\bigg)^{2/M}\bigg), \quad z\rightarrow \infty.\label{qlim3}  
	\end{align}
\end{remark}

\begin{remark}
Part (i) of Theorem \ref{thm2.4} tells us that, as $z\rightarrow0$, the PDF $f_Z(z)$ is bounded for $M=1,N\geq0$, but has a singularity as $z\rightarrow0$ in the cases $M\geq2, N\geq0$ and $M=0, N\geq1$. Since the PDF is bounded everywhere except for the singularity at the origin (for the values of $M$ and $N$ for which there is a singularity), the distribution of $Z$ is unimodal with mode 0 when $M\geq1$, $N\geq0$. On the other hand, in the case $M=0$, $N\geq1$, the limiting form (\ref{2.42}) tells us that, $f_Z(z)\rightarrow0$, as $z\rightarrow0$, meaning that the distribution of $Z$ is not unimodal (since the the PDF $f_Z(z)$ is symmetric about the origin).

It follows from the limiting form (\ref{pdflim}) that the mean of the ratio $Z$ is not defined for $M\geq0$, $N\geq1$. In the case $M=N=1$, we recover the well-known fact that the mean of the Cauchy distribution is undefined. However, the fractional moments $\mathbb{E}[|Z|^r]$ exist for $-1<r<1$, as seen in the following proposition.

\end{remark}

\begin{proposition} Suppose $M\geq0$, $N\geq0$ with $M+N\geq1$. Then, for $-1<r<1$,
\begin{equation}\label{frac}
\mathbb{E}[|Z|^r]=\frac{2^{(M-N)r/2}\sigma_M^r}{\pi^{(M+N)/2}s_N^r}\bigg(\Gamma\Big(\frac{r+1}{2}\Big)\bigg)^M\bigg(\Gamma\Big(\frac{1-r}{2}\Big)\bigg)^N .  
\end{equation}    
When $M\geq1$ and $N=0$, equation (\ref{frac}) holds for all $r>-1$, whilst when $M=0$ and $N\geq1$, equation (\ref{frac}) holds for all $r<1$.
\end{proposition}

\begin{proof} By independence, $\mathbb{E}[|Z|^r]=(\prod_{i=1}^M\mathbb{E}[|X_i|^r])(\prod_{j=1}^N\mathbb{E}[|Y_j|^{-r}])$, and the result now follows from the standard formula $\mathbb{E}[|V|^p]=2^{p/2}\sigma^p \pi^{-1/2}\Gamma((p+1)/2)$ for $p>-1$, where $V\sim N(0,\sigma^2)$.
\end{proof}

\begin{proof}[Proof of Theorem \ref{thm2.4}]
	\begin{comment}
	(i)	In this proof we will make use of the standard asymptotic expansion
	\begin{align}
		\Gamma(s)=\frac{1}{s}+\gamma +o(1), \quad s\rightarrow0,\label{gamma}
	\end{align}
	where $\gamma$ is the Euler-Mascheroni constant, and the Taylor expansion of $x^{-2s}$ about $s=0$ given by
	\begin{align}
		x^{-2s} =\exp(-2s\ln(x)) = \sum^\infty_{i=0} \frac{(-2 \ln (x))^i }{i!} s^i. \label{xexp}
	\end{align}
\end{comment}
(i) 
%Recall that the PDF (\ref{2.1}) consists of a single Meijer $G$-function. 
To simplify expressions, we let $x=2^{(N-M)/2}s_Nz/\sigma_M$. It suffices to consider the limit $x\downarrow0$, since the PDF (\ref{2.1}) is an even function in $z$. Using the contour integral definition of the Meijer-$G$ function (\ref{mdef}) followed by the residue theorem gives that
	\begin{align}
		G^{M,N}_{N,M}\bigg(x^2 \; \bigg| \;{0,\ldots,0 \atop 0,\ldots,0}\bigg)&=\frac{1}{2\pi\mathrm{i}}\int_L\big(\Gamma(s)\big)^M\big(\Gamma(1-s)\big)^Nx^{-2s}\,\mathrm{d}s\nonumber \\
	%G^{M,N}_{N,M}\bigg(x^2 \; \bigg| \;{0,\ldots,0 \atop 0,\ldots,0}\bigg)
    &=\sum_{s^*\in A}\mathrm{Res}\big[\big(\Gamma(s)\big)^M\big(\Gamma(1-s)\big)^Nx^{-2s}, \,s=s^*\big]\label{res1},
\end{align}
		where the contour $L$ is a loop encircling the poles of the gamma functions and $A$ is the collection of the poles of the gamma functions. We now compute the residues in (\ref{res1}), starting with the pole at $s=0$, which amounts to finding the coefficient of $s^{-1}$ in the Laurent expansion of $(\Gamma(s))^Mx^{-2s}$.
		A simple calculation using the expansions $\Gamma(s)=1/s+\gamma+o(1)$, as $s\rightarrow0$, where $\gamma$ is the Euler-Mascheroni constant, and $x^{-2s}%=\exp(-2s\ln(x))
        =\sum_{i=0}^{\infty}(-2\ln(x))^is^i/{i!}$ reveals that
	%	\begin{equation*}
	%		\bigg(\frac{1}{s}+\gamma+o(1)\bigg)^M\sum_{k=0}^{\infty}\frac{(-2\ln(x))^k}{k!}s^k,
	%	\end{equation*}
	%	and it is easy to see that 
		\begin{equation*}
			\mathrm{Res}\big[\big(\Gamma(s)\big)^M\big(\Gamma(1-s)\big)^Nx^{-2s},\,s=0\big]\sim\frac{(-2\ln(x))^{M-1}}{(M-1)!},\quad x\downarrow0.
		\end{equation*}
		For the other poles, the residue is readily seen to be of a smaller asymptotic order in the limit $x\downarrow0$ than the residue at $s=0$. Hence, as $x\rightarrow0$ (equivalently $z\rightarrow0$),
		\begin{align}
			G^{M,N}_{N,M}\bigg(x^2 \; \bigg| \;{0,\ldots,0 \atop 0,\ldots,0}\bigg)&\sim\frac{(-2\ln|x|)^{M-1}}{(M-1)!}\sim\frac{(-2\ln|z|)^{M-1}}{(M-1)!}\label{glim2},
		\end{align}
		from which we deduce the limiting form (\ref{plim1}).
	%	\begin{equation*}
	%		f_Z(z)\sim\frac{2^{(N-M)/2}\sigma_Y}{\pi^{(M+N)/2}\sigma_X}\frac{(-2\ln|z|)^{M-1}}{(M-1)!}.
	%	\end{equation*}
    
%Finally, we verify that the distribution of $Z$ is unimodal with mode 0 when $M\geq1$ and $N\geq0$. From the limiting form (\ref{plim1}), we see that the PDF of $Z$ has a singularity at $z = 0$. Since the PDF is bounded everywhere except for the singularity at the origin, the distribution of Z is unimodal with mode 0.    

\vspace{2mm}
    
	\noindent (ii) We now suppose $M=0$; in this case we set $x=2^{N/2}s_Nz$. Then by identity (\ref{mgiv}) and the limiting form (\ref{asymg}) we obtain that
		\begin{align}
			G^{0,N}_{N,0}\bigg(x^2 \; \bigg| \;{0,\ldots,0 \atop -}\bigg)&=G^{N,0}_{0,N}\bigg(x^{-2} \; \bigg| \;{- \atop 1,\ldots,1}\bigg)\sim\frac{(2\pi)^{(N-1)/2}}{\sqrt{N}}|x|^{-(N+1)/N}\exp\big(-N|x|^{-2/N}\big),\nonumber
			\end{align}
      as $x\rightarrow0$.      The limiting form (\ref{2.42}) now follows on setting $x=2^{N/2}s_Nz$.
\end{proof}

\begin{proof}[Proof of Theorem \ref{thmtail}] (i) By applying identity (\ref{mgiv}) followed by (\ref{meijergidentity}) we have that
\begin{align}
f_Z(z)&=\frac{2^{(N-M)/2}s_N}{\pi^{(M+N)/2}\sigma_M}G^{M,N}_{N,M}\bigg(\frac{2^Ns_N^2}{2^M\sigma_M^2}z^2 \; \bigg| \;{0,\ldots,0 \atop 0,\ldots,0}\bigg)\nonumber\\
&=\frac{2^{(N-M)/2}s_N}{\pi^{(M+N)/2}\sigma_M}G^{N,M}_{M,N}\bigg(\frac{2^M\sigma_M^2}{2^Ns_N^2}\frac{1}{z^2} \; \bigg| \;{1,\ldots,1 \atop 1,\ldots,1}\bigg)\nonumber\\
&=\frac{2^{(N-M)/2}s_N}{\pi^{(M+N)/2}\sigma_M}u^2G^{N,M}_{M,N}\bigg(u^2 \; \bigg| \;{0,\ldots,0 \atop 0,\ldots,0}\bigg),\label{mg1}
%\frac{2^{(M-N)/2}\sigma_X}{\pi^{(N+M)/2}\sigma_Y}\frac{1}{z^2}G^{N,M}_{M,N}\bigg(x^2 \; \bigg| \;{0,\ldots,0 \atop 0,\ldots,0}\bigg), \quad x\rightarrow0,\label{mg1}
\end{align}
where $u=2^{(M-N)/2}\sigma_M/(s_Nz)$.   The limiting form (\ref{pdflim}) now follows from applying the limiting form (\ref{glim2}) to (\ref{mg1}). 

\vspace{2mm}

\noindent (ii) By (\ref{pdflim}), we have that, as $z\rightarrow\infty$,
	\begin{align*}
		\overline{F}_Z(z)=\int_z^\infty f_Z(x)\,\mathrm{d}x\sim\frac{2^{(M+N-2)/2}\sigma_M}{\pi^{(M+N)/2}(N-1)!s_N}\int_{z}^{\infty}\frac{(\ln(x))^{N-1}}{x^2}\,\mathrm{d}x,
	\end{align*}
	and since $\int_{z}^{\infty}(\ln(x))^{N-1}/{x^2}\,\mathrm{d}x\sim(\ln(z))^{N-1}/z$, as $z\rightarrow\infty$, we obtain (\ref{CDFf1}).
%We therefore obtain
%\begin{align}
%f_Z(z)&\sim\frac{2^{(N-M)/2}s_N}{\pi^{(M+N)/2}\sigma_M}x^2\frac{(-2\ln|x|)^{N-1}}{(N-1)!},\quad x\rightarrow0,\nonumber\\
%&\sim\frac{2^{(M-N)/2}\sigma_M}{\pi^{(N+M)/2}s_N}\frac{1}{z^2}\frac{(2\ln|z|)^{N-1}}{(N-1)!}, \quad |z|\rightarrow\infty.\nonumber
%\end{align}
\begin{comment}
2. When $N=0$ and $M\geq1$, we can use (\ref{glim2}) with $M$ replaced by $N$, to find the asymptotic behaviour of the Meijer $G$-function in (\ref{mg1}). We have
\begin{align}
	G^{0,M}_{M,0}\bigg(x^2 \; \bigg| \;{0,\ldots,0 \atop -}\bigg)&\sim\frac{(2\pi)^{(M-1)/2}}{\sqrt{M}}x^{-(M+1)/M}\exp(-Mx^{-2/M}),\quad x\rightarrow0\nonumber\\
	&\sim\frac{\pi^{(M-1)/2}}{2\sqrt{M}}\bigg(\frac{z}{\sigma_X}\bigg)^{(M+1)/M}\exp\bigg(-M\bigg(\frac{z^2}{2^M\sigma_X^2}\bigg)^{1/M}\bigg),\quad z\rightarrow\infty.\label{mg2}
\end{align}
Substituting (\ref{mg2}) into (\ref{mg1}) with $N=0$, we obtain (\ref{2.52}).
\end{comment}
\end{proof}

To prove Theorem \ref{qprop} we use the following lemma which is a special case of \cite[Lemma 4.2]{gz24}.
%with $b=0$.
\begin{lemma}
	Let $a, A, z > 0$ and $m\in\mathbb{R}$. Let $g:(0,\infty)\rightarrow\mathbb{R}$ be a function such that $g(x)\rightarrow0$ as $x\rightarrow\infty$. Consider the solution $x$ of the equation
	\begin{equation}
		Ax^{m}\mathrm{e}^{-ax}\big(1+g(x)\big)=z,\label{q1}
	\end{equation}
	and note that there exists a unique solution as long as $z$ is sufficiently small. Then
    %, as $z\rightarrow0$,
	\begin{align}
		x&=\frac{1}{a}\ln(1/z)+\frac{m}{a}\ln\big(\ln(1/z)\big)+\frac{\ln(A/a^m)}{a}+o(1), \quad z\rightarrow0.\label{q2}
	\end{align}
\end{lemma}
\noindent{\emph{Proof of Theorem \ref{qprop}.}} (i) %We derive the limiting form (\ref{limf1}). 
Since the quantile function $Q(p)$ solves the equation $F_Z(Q(p))=p$ (and recalling that $F_Z(-z)=\overline{F}_Z(z)$), by applying the limiting form (\ref{CDFf1}) it can be seen that $Q(p)$ solves an equation of the form (\ref{q1}) with $x=\ln(-Q(p))$, and $z=p$, $a=1$, $m=N-1$ and $A=2^{(M+N-2)/2}\sigma_M/(\pi^{(M+N)/2}(N-1)!s_N)$. By applying (\ref{q2}) with $a=1$, $z=p$, $m=N-1$ and $x=\ln(-Q(p))$, we get that, as $p\rightarrow0$,
\begin{align}
	Q(p)&=-\exp\Big(\ln(1/p)+\ln\big((\ln(1/p))^{N-1}\big)+\ln(A)+o(1)\Big)\sim-\frac{A}{p}\big(\ln(1/p)\big)^{N-1}.\label{y1}
\end{align}
The limiting form (\ref{limf2}) now follows from inputting the above value of $A$ into (\ref{y1}).
%The derivation of the limiting form (\ref{limf2}) is similar and therefore omitted. 

\vspace{2mm}

\noindent (ii) Upon applying the limiting form (\ref{qlim3}) to the equation $F_Z(Q(p))=p$ (and again recalling that $F_Z(-z)=\overline{F}_Z(z)$), it can be seen that $Q(p)$ solves an equation of the form (\ref{q1}) with $x=(-Q(p))^{2/M}$, and $A=(2^{M/2-1}\sigma_M^{1/M})/\sqrt{M\pi}$, $z=p$, $a=M/(2\sigma_M^{2/M})$ and $m=-1/2$. By applying (\ref{q2}) with $m=-1/2$, $z=p$ and $x=(-Q(p))^{2/M}$, we have, as $p\rightarrow0$, 
\begin{align}
	Q(p)&=-\bigg(\frac{1}{a}\ln(1/p)-\frac{1}{2a}\ln\big(\ln(1/p)\big)+\frac{\ln(Aa^{1/2})}{a}+o(1)\bigg)^{M/2} \sim-\bigg(\frac{1}{a}\ln(1/p)\bigg)^{M/2}, \label{y2}
\end{align}
and the limiting form (\ref{limf4}) follows by inputting $a=M/(2\sigma_M^{2/M})$  into (\ref{y2}).
 \qed
\begin{comment}
\noindent{\emph{Proof of Proposition \ref{qprop}.}} 1. We derive the limiting form (\ref{limf1}). Since the quantile function $Q(p)$ solves the equation $\overline{F}_Z(Q(p))=1-p$, upon applying the limiting form (\ref{CDFf1}) it can be seen that $Q(p)$ solves an equation of the form (\ref{qeqn}) with $A=2^{(M+N)/2-1}\sigma_X/(\pi^{(M+N)/2}\sigma_Y(N-1)!),$ $z=1-p$, $a=1$ and $m=N-1$. We therefore deduce the limiting form (\ref{limf1}) by applying (\ref{qf1}) with these values of $A$, $z$, $a$ and $m$.
The derivation of the limiting form (\ref{limf2}) is similar and is therefore omitted. 

2.  The derivation of the limiting forms (\ref{limf3}) and (\ref{limf4}) is similar to the above; we use the limiting forms (\ref{qlim3}) and (\ref{qlim4}), respectively along with Lemma \ref{l4}. \qed
\end{comment}

\section{Special cases}\label{sec3}
In this section, we present some special cases in which the general formulas for the PDF, CDF and characteristic function from Section \ref{sec2} take a simpler form. 

%we list the PDFs of some special cases of our distribution. 

\vspace{2mm}

\noindent 1. $N=0$: As mentioned in the introduction, in the case $M=2$, the PDF can be expressed in terms of the modified Bessel function of the second kind. In the $M=2$ case, a closed-form formula in terms of the modified Bessel function of the second kind and the modified Struve function of the first kind is also available; see \cite[Section 2.3]{gaunt22}. For $M=2,3,4$, simpler formulas, expressed in terms of elementary functions and Bessel functions, are available for the characteristic function; see \cite[Remark 3.8]{gl24}. 

\vspace{2mm}

\noindent 2. $M=0$, $N=1$: the random variable $Z=1/Y$ follows the reciprocal normal distribution with PDF $f_Z(z)=(\sqrt{2\pi}s_1z^2)^{-1}\exp(-1/(2s_1^2z^2))$, $z\in\mathbb{R}$. We can apply the reduction formula (\ref{M0N1}) to the general formula (\ref{ratioo}) for the characteristic function to obtain the following formula, which we could not find in the literature, for the characteristic function of $Z=1/Y$ in terms of the ${}_0F_2
%(;a,b;x)
$ generalized hypergeometric function (see \cite[Chapter 16]{olver}):
\begin{align*}
\varphi_Z(t)= {}_0F_2\bigg(;\frac{1}{2},\frac{1}{2};\frac{t^2}{8s_1^2}\bigg)-\frac{\sqrt{\pi}|t|}{\sqrt{2}s_1}\,{}_0F_2\bigg(;1,\frac{3}{2};\frac{t^2}{8s_1^2}\bigg), \quad t\in\mathbb{R}. 
\end{align*}

%where ${}_0F_2(;a,b;x)$ is a generalized hypergeometric function.

%\vspace{2mm}

\noindent 3. $M=0$, $N=2$: By the standard formula for the PDF of a one-to-one transformation of a random variable, we have $f_Z(z)=f_{1/Y}(z)=z^{-2}f_Y(1/z)$, $z\in\mathbb{R}$, and therefore 
\begin{align*}
f_Z(z)=\frac{1}{\pi s_2z^2}K_0\bigg(\frac{1}{s_2|z|}\bigg),\nonumber\quad z\in\mathbb{R}.   \end{align*}
Also, since the distribution of $Z$ is symmetric about the origin
%, we have that $F_Z(z)=1/2+\mathrm{sgn}(z)\int_0^{|z|} f_Z(x)\,\mathrm{d}x$.
%Therefore, 
we have, for $z\in\mathbb{R}$,
\begin{align*}
F_Z(z)&=  \frac{1}{2}  +\frac{\mathrm{sgn}(z)}{\pi s_2}\int_0^{|z|} \frac{1}{x^2}K_0\bigg(\frac{1}{s_2x}\bigg)\,\mathrm{d}x= \frac{1}{2}  +\frac{\mathrm{sgn}(z)}{\pi}\int_0^{1/(s_2|z|)} K_0(u)\,\mathrm{d}u\\
&=\frac{1}{2}+\frac{1}{2s_2z}\bigg[K_0\bigg(\frac{1}{s_2|z|}\bigg)\mathbf{L}_{-1}\bigg(\frac{1}{s_2|z|}\bigg)+\mathbf{L}_0\bigg(\frac{1}{s_2|z|}\bigg)K_1\bigg(\frac{1}{s_2|z|}\bigg)\bigg],
\end{align*}
where $\mathbf{L}_\nu(x)$ is the modified Struve function of the first kind \cite[Chapter 11]{olver}. Here we evaluated the integral using \cite[equation 6.561(4)]{gradshetyn} together with the basic identity $K_{-1}(x)=K_1(x)$ (see \cite[equation 10.27.3]{olver}).

\vspace{2mm}

%\noindent1. Setting $M=N=1$ in the PDF formula (\ref{2.1}) and using (\ref{MG1}), we recover the PDF of the Cauchy distribution:
%\begin{equation}
	%f_Z(z)=\frac{1}{\tfrac{\sigma_1}{s_1}\pi\big(1+\tfrac{s_1^2}{\sigma^2_1}z^2\big)},\nonumber\quad z\in\mathbb{R}.
%\end{equation}
\noindent 4. $M=N$: This case corresponds to the product of $M$ independent Cauchy random variables.
%, or equivalently the ratio of two independent Cauchy random variables in the case $M=N=2$. 
Applying the reduction formulas in (\ref{MG3}) to the PDF formula (\ref{2.1}) yields the following elementary formulas for the PDFs of $Z_2={X_1X_2}/({Y_1Y_2})$ and $Z_3={X_1X_2X_3}/({Y_1Y_2Y_3})$:
\begin{align}
	\label{vgy} f_{Z_2}(z)&=\frac{2r_2\ln(r_2|z|)}{\pi^2(r_2^2z^2-1)},\quad z\in\mathbb{R},\\
   \label{vgy2} f_{Z_3}(z)&=\frac{r_3(4\ln^2(r_3|z|)+\pi^2)}{2\pi^3(r_3^2z^2+1)},\quad z\in\mathbb{R},
\end{align}
where $r_2=s_2/\sigma_2$ and $r_3=s_3/\sigma_3$. Formula (\ref{vgy}) was earlier derived by \cite{r65} (in the case $\sigma_2=s_2=1$). Formula (\ref{vgy2}) was obtained by \cite{st66} (in the case $\sigma_3=s_3=1$), who more generally gave an expression for the PDF of the product of $M$ independent Cauchy random variables as a finite series expressed in terms of powers of $\ln(|z|)$, and gave closed-form formulas for $M=1,2,\ldots,10$. 

A closed-form formula, expressed in terms of the dilogarithm function, for the CDF of $Z_2$ can also be given (such a formula was not given in \cite{r65} or \cite{st66}). The dilogarithm is a special case of the polylogarithm, which can be defined by the power series $\mathrm{Li}_s(x)=\sum_{k=0}^\infty x^k/k^s$ for $|x|<1$ and by analytic continuation elsewhere; the dilogarithm $\mathrm{Li}_2$ corresponds to the case $s=2$. Since the distribution of $Z$ is symmetric about the origin, we have that $F_{Z_2}(z)=1/2+\mathrm{sgn}(z)\int_0^{|z|} f_{Z_2}(x)\,\mathrm{d}x$. Thus, for $z\in\mathbb{R}$,
\begin{align}
F_{Z_2}(z)&=\frac{1}{2}+\frac{2\,\mathrm{sgn}(z)}{\pi^2}\int_0^{r_2|z|}\frac{\ln(x)}{x^2-1}\,\mathrm{d}x=\frac{1}{2}+\frac{\mathrm{sgn}(z)}{\pi^2}\int_0^{r_2|z|}\bigg(\frac{\ln(x)}{x-1}-\frac{\ln(x)}{x+1}\bigg)\,\mathrm{d}x\nonumber\\
&=\frac{1}{2}+\frac{\mathrm{sgn}(z)}{\pi^2}\int_0^{r_2|z|}\bigg(\frac{\ln(x)}{x-1}+\frac{\ln(x+1)}{x}-\bigg(\frac{\ln(x)}{x+1}+\frac{\ln(x+1)}{x}\bigg)\bigg)\,\mathrm{d}x\nonumber \\
&=\frac{1}{2}+\frac{\mathrm{sgn}(z)}{\pi^2}\bigg[\frac{\pi^2}{6}-\mathrm{Li}_2(1-r_2|z|)-\mathrm{Li}_2(-r_2|z|)-\ln(r_2|z|)\ln(1+r_2|z|)\bigg],   \label{redd}
\end{align}
where we evaluated the integral using the integral formula $\int_1^x\ln(t)/(1-t) \,\mathrm{d}t=\mathrm{Li}_2(1-x)$ (see \cite[equation 25.12.2]{olver}), the basic derivative $\frac{\mathrm{d}}{\mathrm{d}x}(\ln(x)\ln(x+1))=\log(x)/(x+1)+\log(x+1)/x$, and the specific values $\mathrm{Li}_2(1)=\pi^2/6$ and $\mathrm{Li}_2(0)=0$.
A closed-form formula, expressed in terms of the polylogarithm function, can also be given for the CDF of $Z_3$; however, the formula is rather complicated, so we do not report it.

Finally, we note that as a by-product of this analysis we obtain the following reduction formula for the Meijer $G$-function, which to the best of our knowledge is new.

\begin{corollary}Suppose $a\in\mathbb{R}$. Then, for $x>0$,
\begin{align*}
G_{3,3}^{2,3}\bigg(x \; \bigg| \;{a+\frac{1}{2},a,a\atop a,a,a-\frac{1}{2}}\bigg)= 2x^{a-1/2}\bigg[\frac{\pi^2}{6}-\mathrm{Li}_2(1-\sqrt{x})-\mathrm{Li}_2(-\sqrt{x})-\ln(\sqrt{x})\ln(1+\sqrt{x})\bigg].  
\end{align*}
\end{corollary}

\begin{proof}
Combine equations (\ref{2.2b}) and (\ref{redd}) and then apply the identity (\ref{meijergidentity}).
\end{proof}

\noindent 5. $M=2$, $N=1$: Combining formulas (\ref{2.1}) and (\ref{MG5}), and formulas (\ref{2.3b}) and (\ref{nnnn}) yields the following formulas for the PDF and characteristic function of $Z=X_1X_2/Y_1$:
\begin{align}\label{nonn}
	f_Z(z)&=\frac{s_1}{\sqrt{2}\pi^{3/2}\sigma_2}\exp\bigg(\frac{s_1^2z^2}{2\sigma_2^2}\bigg)E_1\bigg(\frac{s_1^2z^2}{2\sigma_2^2}\bigg), \quad z\in\mathbb{R}, \\
    \varphi_Z(t)&=\frac{1}{\sqrt{\pi}}\exp\bigg(\frac{\sigma_2^2t^2}{2s_1^2}\bigg)\Gamma\bigg(\frac{1}{2},\frac{\sigma_2^2t^2}{2s_1^2}\bigg),\quad t\in\mathbb{R}, \nonumber
\end{align}
where $E_1(x)=\int_{0}^{\infty}\mathrm{e}^{-t}/t\,\mathrm{d}t$ is the exponential integral function and $\Gamma(a,x)=\int_x^\infty u^{a-1}\mathrm{e}^{-u}\,\mathrm{d}u$ is the upper incomplete gamma function.

\vspace{2mm}

\noindent 6. $M=1$, $N=2$: From formula (\ref{nonn}) we readily obtain the following formula for the PDF of $Z={X_1}/({Y_1Y_2})(=(Y_1Y_2/X_1)^{-1})$:
\begin{equation}
	f_Z(z)=\frac{\sigma_1}{\sqrt{2}\pi^{3/2}s_2z^2}\exp\bigg(\frac{\sigma_1^2}{2s_2^2z^2}\bigg)E_1\bigg(\frac{\sigma_1^2}{2s_2^2z^2}\bigg),\nonumber\quad z\in\mathbb{R}.
\end{equation}

\vspace{2mm}

\noindent 7. $M=3$, $N=1$: In this case, we were unable to locate in the literature a reduction formula for the Meijer $G$-function that allows for a simpler formula to be given for the PDF. However, on applying the reduction formula (\ref{mmmm}) to (\ref{2.3b}) we obtain the following formula for the characteristic function of $Z=X_1X_2X_3/Y_1$:
\begin{align*}
\varphi_Z(t)= \frac{2\cos^{-1}(\sigma_3|t|/s_1)}{\pi\sqrt{1-(\sigma_3|t|/s_1)^2}}, \quad t\in\mathbb{R}.   
\end{align*}

%6. Setting $M=0$ and $N=1$ in our PDF formula (\ref{2.1}) and using (\ref{MG6}), we obtain the following PDF for $Z={1}/{Y}$, where $Y\sim N(0,s_1^2)$
%\begin{equation}
	%f_Z(z)=\frac{1}{\sqrt{2\pi}s_1z^2}\exp\bigg(-\frac{1}{2s_1^2z^2}\bigg),\nonumber \quad z\in\mathbb{R}.
%\end{equation}
%7. Setting $M=1$ and $N=0$ in our PDF formula (\ref{2.1}) and using (\ref{MG7}) yields the $N(0,\sigma_1^2)$ PDF.

%\noindent8. Setting $M=0$ and $N=2$ in the PDF formula (\ref{2.1}) and applying (\ref{MG8}), we have the following PDF for $Z={1}/{(Y_1Y_2)}$, where $Y_j\sim N(0, \sigma^2_{Y_j}), \, j=1,2$
%\begin{equation}
%	f_Z(z)=\frac{1}{\pi s_2z^2}K_0\bigg(\frac{1}{s_2z}\bigg),\nonumber\quad z\in\mathbb{R}.
%\end{equation}

%\noindent9. Setting $M=2$ and $N=0$ in our PDF formula (\ref{2.1}) and applying (\ref{MG9}), we obtain the following PDF for the product of two independent normal random variables with zero means
%\begin{equation}
	%f_Z(z)=\frac{1}{\pi\sigma_2}K_0\bigg(\frac{z}{\sigma_2}\bigg),\nonumber \quad z\in\mathbb{R},
%\end{equation}
%which is in agreement with the formula (\ref{1.1}) when $M=2$.

%\section{A reduction formula for the Meijer $G$-function}\label{sec4}

\appendix
\section{The Meijer $G$-function}\label{appa}
In this appendix, we define the Meijer $G$-function, and state some of its relevant basic properties, which, unless otherwise stated, can be found in the standard references \cite{luke,olver}.

The \emph{Meijer $G$-function} is defined, for $x\in\mathbb{R}$, by the contour integral
\begin{equation}\label{mdef}G^{m,n}_{p,q}\bigg(x \, \bigg|\, {a_1,\ldots, a_p \atop b_1,\ldots,b_q} \bigg)=\frac{1}{2\pi \mathrm{i}}\int_L\frac{\prod_{j=1}^m\Gamma(b_j-s)\prod_{j=1}^n\Gamma(1-a_j+s)}{\prod_{j=n+1}^p\Gamma(a_j-s)\prod_{j=m+1}^q\Gamma(1-b_j+s)}x^s\,\mathrm{d}s,
\end{equation}
where the integration path $L$ separates the poles of the factors $\Gamma(b_j-s)$ from those of the factors $\Gamma(1-a_j+s)$.  We use the convention that the empty product is $1$.

The $G$-function satisfies the identities
\begin{align}\label{lukeformula}G_{p,q}^{m,n}\bigg(x \; \bigg| \;{a_1,\ldots,a_{p-1},b_1 \atop b_1,\ldots,b_q}\bigg)&=G_{p-1,q-1}^{m-1,n}\bigg(x \; \bigg| \;{a_1,\ldots,a_{p-1} \atop b_2,\ldots,b_q}\bigg), \quad n>p;\, m,p,q\geq 1, \\
\label{meijergidentity}x^\alpha G_{p,q}^{m,n}\bigg(x \; \bigg| \;{a_1,\ldots,a_p \atop b_1,\ldots,b_q}\bigg)&=G_{p,q}^{m,n}\bigg(x \, \bigg| \,{a_1+\alpha,\ldots,a_p+\alpha \atop b_1+\alpha,\ldots,b_q+\alpha}\bigg),\\
	G_{p,q}^{m,n}\bigg(x \; \bigg| \;{a_1,\ldots,a_p \atop b_1,\ldots,b_q}\bigg)&=G_{q,p}^{n,m}\bigg(x^{-1} \; \bigg| \;{1-b_1,\ldots,1-b_q \atop1-a_1,\ldots,1-a_p}\bigg).\label{mgiv}
\end{align}
Combining equations 5.4(1) and 5.4(13) of \cite{luke}, yields the following indefinite integral formula:
\begin{equation}\label{mint}\int x^{\alpha-1}G_{p,q}^{m,n}\bigg(x \; \bigg| \;{a_1,\ldots,a_p \atop b_1,\ldots,b_q}\bigg)\,\mathrm{d}x=x^\alpha G_{p+1,q+1}^{m,n+1}\bigg(x \, \bigg| \,{1-\alpha,a_1,\ldots,a_p \atop b_1,\ldots,b_q,-\alpha}\bigg).
\end{equation}
%We also have the integral formula
%\begin{equation}\label{mint1}\!\!\!\int G_{p,q}^{m,n}\bigg(x \, \bigg| \,{a_1,\ldots,a_p \atop b_1,\ldots,b_q}\bigg)\,\mathrm{d}x= G_{p+1,q+1}^{m,n+1}\bigg(x \, \bigg| \,{1,a_1+1,\ldots,a_p+1 \atop b_1+1,\ldots,b_m+1,0,b_{m+1}+1,\ldots, b_q+1}\bigg).
%\end{equation}
The following limiting form is given in \cite[Section 5.7, Theorem 5]{luke}: as $x\rightarrow\infty$,
\begin{equation}\label{asymg}G^{q,0}_{p,q}\bigg(x \; \bigg|\; {a_1,\ldots, a_p \atop b_1,\ldots,b_q} \bigg)\sim \frac{(2\pi)^{(\sigma-1)/2}}{\sqrt{\sigma}}x^\theta \exp\big(-\sigma x^{1/\sigma}\big), 
\end{equation}
where $\sigma=q-p$ and $\theta=\sigma^{-1}\{(1-\sigma)/2+\sum_{i=1}^qb_i-\sum_{i=1}^pa_i\}$.

We have the following reduction formulas:
\begin{align}
	%G_{1,1}^{1,1}\bigg(x \; \bigg| \;{0\atop 0}\bigg)&=\frac{1}{x+1},\label{MG1}\\
	G_{2,2}^{2,2}\bigg(x \; \bigg| \;{0,0\atop 0,0}\bigg)&=\frac{\ln(x)}{x-1},	\quad
	G_{3,3}^{3,3}\bigg(x \; \bigg| \;{0,0,0\atop 0,0,0}\bigg)=\frac{\ln^2(x)+\pi^2}{2(x+1)},\label{MG3}\\
	%G_{2,1}^{1,2}\bigg(x \; \bigg| \;{0,0\atop 0}\bigg)&=\frac{\exp(1/x)}{x}E_1\big(1/x\big),\label{MG4}\\
	G_{1,2}^{2,1}\bigg(x \; \bigg| \;{0\atop 0,0}\bigg)&=\mathrm{e}^xE_1(x),\label{MG5}\\
	%G_{1,0}^{0,1}\bigg(x \; \bigg| \;{0\atop -}\bigg)&=\frac{\exp(-1/x)}{x},\label{MG6}\\
	%G_{0,1}^{1,0}\bigg(x \; \bigg| \;{-\atop 0}\bigg)&=\exp(-x),\label{MG7}
	%G_{2,0}^{0,2}\bigg(x \; \bigg| \;{0,0\atop -}\bigg)&=\frac{2}{x}K_0\bigg(\frac{2}{\sqrt{x}}\bigg),\label{MG8}\\
	%G_{0,2}^{2,0}\bigg(x \; \bigg| \;{-\atop 0,0}\bigg)&=2K_0(2\sqrt{x}),\label{MG9}
    G_{1,2}^{2,1}\bigg(x \; \bigg| \;{\frac{1}{2}\atop 0,\frac{1}{2}}\bigg)&=\sqrt{\pi}\mathrm{e}^x\Gamma\bigg(\frac{1}{2},x\bigg),\label{nnnn}\\
    G_{2,2}^{2,2}\bigg(x \; \bigg| \;{\frac{1}{2},\frac{1}{2}\atop 0,\frac{1}{2}}\bigg)&=\sqrt{\frac{\pi}{x}}\Phi\bigg(1-\frac{1}{x},1,\frac{1}{2}\bigg) %\label{oooo} \\
    =\frac{2\sqrt{\pi}\cos^{-1}(\sqrt{x})}{\sqrt{1-x}},\label{mmmm}\\
    G_{0,3}^{2,0}\bigg(x \; \bigg| \;{-\atop 0,\frac{1}{2},\frac{1}{2}}\bigg)&={}_0F_2\bigg(;\frac{1}{2},\frac{1}{2};x\bigg)-2\sqrt{\pi x}\,{}_0F_2\bigg(;1,\frac{3}{2};x\bigg),\label{M0N1}
\end{align}
where these reduction formulas are suitable special cases of formulas that can be found at http://functions.wolfram.com/07.34.03.0887.01, http://functions.wolfram.com/07.34.03.1053.01, http://functions.wolfram.com/07.34.03.0724.01, http://functions.wolfram.com/07.34.03.0723.01, http://functions.wolfram.com/07.34.03.0880.01, http://functions.wolfram.com/07.34.03.0607.01, respectively. Here $\Phi(x,s,a)$ is the Lerch transcendent, which is defined by the power series $\Phi(x,s,a)=\sum_{k=0}^\infty x^k/(a+k)^s$ for $|x|<1$ and by analytic continuation elsewhere. We obtained the second equality in (\ref{mmmm}) by using the identities $\Phi(x,1,1/2)=2x^{-1/2}\tanh^{-1}(\sqrt{x})$ and $\tanh^{-1}(\sqrt{1-1/x})=-\mathrm{i}\cos^{-1}(\sqrt{x})$.
%where 
%$K_0$ is a modified Bessel function of the second kind and 
%$E_1$ is the exponential integral function defined by $E_1(x)=\int_{0}^{\infty}\mathrm{e}^{-t}/t\,\mathrm{d}t$ and ${}_0F_2$ is a generalized hypergeometric function.

\section*{Acknowledgements}
RG was funded in part by EPSRC grant EP/Y008650/1 and EPSRC grant UKRI068. HS was funded by EPSRC grant EP/Y008650/1.

%\section*{Conflict of interest statement}

%There is no conflict of interest.

\end{document}